\documentclass{article}

\usepackage{amsmath,amssymb}
\usepackage{latexsym}
\usepackage{pb-diagram}

\begin{document}

\newtheorem{thm}{Theorem}[section]
\newtheorem{lem}[thm]{Lemma}
\newtheorem{col}[thm]{Corollary}
\newtheorem{df}[thm]{Definition}
\newtheorem{exam}[thm]{Example}
\newtheorem{rem}[thm]{Remark}
\newtheorem{notation}[thm]{Notation}
\newtheorem{claim}{Claim.}
\renewcommand{\theclaim}{}

\newcommand{\N}{\mathbf N}
\newcommand{\nast}{\mathcal N^{\ast}}
\newcommand{\fps}{\mathbf{FPS}}
\newcommand{\f}{\mathbf f}
\newcommand{\z}{\mathrm Z}
\newcommand{\spc}{\mathrm{Sp}}
\newcommand{\minus}{\dot{\text{--}}}
\newcommand{\mprod}{\times^{\minus}}

\newcommand{\dgl}{\underline{\mathrm{dgl}}}
\newcommand{\half}{\underline{\mathrm{half}}}
\newcommand{\pair}{\underline{\mathrm{pair}}}
\newcommand{\fst}{\underline{\mathrm{fst}}}
\newcommand{\snd}{\underline{\mathrm{snd}}}

\newcommand{\sub}{\mathrm{SUB}}
\newcommand{\snrn}{\mathrm{SNRN}}
\newcommand{\RN}{\rightarrow_{R_{\N}}}
\newcommand{\RS}{\rightarrow_{in}}
\newcommand{\Sp}{\mathrm{Sp}}
\newcommand{\lh}{\mathrm{lh}}
\newcommand{\ul}{\underline}

\newenvironment{proof}{\noindent{\bf Proof.}}{\hfill $ \square $}

\newcounter{enumeroman}
\def\theenumeroman{(\roman{enumeroman})}
\newenvironment{enumeroman}
{
  \begin{list}{(\roman{enumeroman})}{\usecounter{enumeroman}}
}
{
  \end{list}
}

\renewcommand{\theequation}{\fnsymbol{equation}}

\title{A term-rewriting characterization of PSPACE}

\author{Naohi Eguchi \\ \\
        Graduate School of Engineering \\
        Kobe University}

\date{January, 2010}

\maketitle

\begin{abstract}
Isabel Oitavem has introduced a term rewriting system (TRS) which captures the
 class \textbf{FPS} of polynomial-space computable functions.
We propose an alternative TRS for \textbf{FPS}.
As a consequence, it is obtained that \textbf{FPS} is the smallest class
 containing certain initial functions and closed under specific operations.
It turns out that our characterization is relatively simple and suggests
an uniform approach to the space-complexity.  
\end{abstract}

\section*{Introduction}

Term rewriting is known as an abstract model of computation, 
since a term is rewritten by successively replacing subterms by equal
terms until no further reduction is possible.
It is also known that term rewriting forms a Turing complete model of
computation.
This paper is an application of term rewriting to small complexity
classes which only involve feasibly computable functions.

The immediate motivation has come from a work by I. Oitavem.
In \cite{oitavem02}, a term rewriting
characterization of the class $\fps$ of functions computed in
polynomial space was given.
This is based on a recursion-theoretic characterization of 
$\fps$ which was obtained by herself. 
In \cite{oitavem97} Oitavem reformulates $\fps$ according to a principle
which was initiated by S. Bellantoni and S. Cook.
The Bellantoni-Cook principle separates variables in every function by
semi-colon as follows:
$f ( \vec x; \vec y)$ -- variables $\vec x$ occurring to the left of the
semi-colon are called normal, while variables $\vec y$ to the right are
called safe.
Roughly speaking, this principle allows recursion only for normal
positions (safe recursion), while composition only for safe positions
(safe composition).
In \cite{bel_cook} a recursion-theoretic characterization of
the polynomial-time computable functions is given
with the use of safe recursion on notation.

Oitavem has shown that it is possible to characterize classes of
computational complexity involving space constraints by techniques from
the field of term rewriting.
Oitavem's function class, however, contains a initial function with large
growth rate like the product, and two recursion schemes.

In this paper we introduce a new term rewriting system for $\fps$.
As a consequence, an alternative characterization of $\fps$ is obtained.
This system is also based on the Bellantoni-Cook principle.
Nevertheless, our class need only elementary initial functions and one
recursion scheme which is called safe nested recursion.
Hence our characterization is relatively simple.
It also turns out that specific safe recursion schemes capture various
space-complexity classes in the presence of the same initial functions
(see Remark \ref{rem_uniform}).
In this sense, our formulation is uniform to the space complexity.

One problem is that translating the safe nested recursion scheme into
the rewriting rule results in handling terms of exponential size.
This is not admissible to capture $\fps$.
We solve this problem by employing the \emph{innermost strategy}. 
Furthermore, for the converse direction, it is not trivial whether the
poly-space computations can be simulated in the present formulation.
A main task is to define a pairing and unpairing functions.

The scheme of safe nested recursion has been introduced by T. Arai and the
author to characterize the 
exponential-time computable functions.
This paper is not fully self-contained, but owes observations on some
properties of safe nested recursion to \cite{arai_eguchi}.

In Section \ref{sec_snrn} we define a function class $\N$ via the
operation of safe nested recursion.
The main definition is given in Section \ref{sec_trs}.
Based on $\N$, we introduce a rewrite system $R_{\N}$ over a class 
$\mathcal F$ of function symbols.
It is shown that for any $f \in \mathcal F$, the length of every term
occurring in a rewriting sequence which starts with $f( \vec t)$ is bounded
by a polynomial in the lengths of input terms $\vec t$
if an innermost strategy is used.
As a corollary, we have $\N \subseteq \fps$.
In Section \ref{sec_pspace}, conversely, it is shown that every
polyspace computation is simulated in $\N$.

\section{A function class $ \N$ with safe nested recursion}
\label{sec_snrn}

\newcommand{\lex}{\prec}
\newcommand{\lexeq}{\preceq}

The first section is devoted to introduce a class $\N$ via the operation of safe nested
recursion on notation (SNRN).
The scheme of SNRN is introduced by Arai and the author
\cite{arai_eguchi}. 

\begin{notation} 
\label{convention}
\normalfont
Although the functions in $\N$ are defined over the natural numbers, 
numbers are denoted as binary strings.
For instance, $2x+i$ is denoted by $xi$ for each $i \in \{ 0, 1 \}$.
Similarly, $xy$ denotes the concatenation of numbers $x$ and $y$ in the
 binary representation.

Fix the signature
$\Sigma = \{ 0, 1, \z \}$ and put
$\Sigma^k := \{ \sigma_1 \cdots \sigma_k : 
                \sigma_1, \dots, \sigma_k \in \Sigma 
             \} \setminus \{ \z \cdots \z \}
$.
If $w \in \Sigma^k$ is of the form $\sigma_1 \cdots \sigma_k$, then 
$w(i)$ denotes $\sigma_i$ for every $i =1, \dots, k$.
Numerals are represented by the binary successor 
$C_0, C_1 \in \N $ such that $C_0 (x; ) = x0$ and $C_1 (x;) = x1$.
Extending this notation to 
$\Sigma = \{ 0, 1, \z \}$, we mean $C_\z (x;)$ for $0$, and 
$C_w ( \vec y;)$ abbreviates 
$(C_{w(1)} (y_1;), \dots, C_{w(k)} (y_k;))$.
For $ \vec y = (y_1, \dots, y_k)$, $w \in \Sigma^k$ and
$i \in \{ k+1, \dots, k+k \}$,
$C_{w(i)} (y_i;)$ denotes $y_{i-k}$ if 
$w(i-k) \in \{ 0, 1 \}$, otherwise $0$.

Let 
$ f ( \vec x, z) [g ( \vec y) / z] $
denote 
$ f ( \vec x, g ( \vec y))$,
the result of a substitution.
$ |x| $ denotes the length of the binary representation of a number $x$, 
i.e., 
$ |x| = \lceil \log_2 (x+1) \rceil $,
which is called the binary length of $x$.
Moreover, for 
$ \vec x = ( x_1 , \dots , x_k ) $, let
$ | \vec x | := ( |x_1| , \dots , |x_k| )$ and
$ \max \vec x := \max \{ x_i : i=1, \dots , k  \} $.
\end{notation}

The computation for every function defined by nested recursion runs
 along the lexicographic ordering.
The scheme of SNRN is defined via a lexicographic ordering $\lex$
and the $ \lex $-functions.

\begin{df}[$\lex$-predecessors]
\normalfont
Suppose $k \geq 1$ and $w \in \Sigma^k$.
For $\vec v = (v_1, \dots, v_k)$ and
$\vec y = (y_1, \dots, y_k)$,
we define $\vec v \lex^k \vec y$.
\begin{enumerate}
\item 
The case $k=1$ is defined by 
$v \lex^1 C_i (v;)$ for an $i \in \{ 0, 1 \}$.
We write $v \lexeq^1 y$ if $v \lex^1 y $ or $v=y$.
\item
For the case $k>1$,
$( v_1, \dots, v_k ) \lex^k (y_1, \dots, y_k)$ 
iff there exists $k_0 \in \{ 1, \dots, k \}$ such that 
$\forall i < k_0 (v_i = y_i)$,
$v_{k_0} \lex^1 y_{k_0}$ and
$\forall i > k_0 \exists j \in \{1, \dots, k\}
 (v_i \lexeq^1 y_j)
$.
\end{enumerate}
If $ \vec v \lex^k \vec y$, then $ \vec v$ is called a
      $\lex^k$-predecessor of $\vec y$.
\end{df}

The ordering $\lex$ is a natural restriction of the usual lexicographic
 ordering in the Bellantoni-Cook principle.
To see this, for now modify $\lex^2$ as
\begin{itemize}
\def\labelitemi{--}
\item $(x, y) \lex^2 (x+1, 0)$ if $y \in \{x, x+1, 0, 1 \}$, 
\item $(x+1, y) \lex^2 (x+1, y+1)$, and
\item $(x, v) \lex^2 (x+1, y+1)$ if $v \in \{x, x+1, y, y+1 \}$.
\end{itemize}
Let us consider Ackermann function here.
Ackermann function $A(x,y)$ is defined by nested recursion on 
$(x,y)$:
$A (0, y) = y+1$,
$A (x+1, 0) = A (x,1)$,
$A (x+1, y+1) = A (x, A (x+1, y))$.
In the equations,
$(x, 1) \lex^2 (x+1, 0)$ and 
$(x+1, y) \lex^2 (x+1, y+1)$, but
$(x, A (x+1, y)) \not\lex^2 (x+1, y+1)$
due to the presence of $A(x+1, y)$.
Recall that the Bellantoni-Cook principle forbids substituting the
 recursion terms into the recursion parameters.
These observations give rise to the definition of a weaker lexicographic
 ordering 
$\lex$.

To determine an operation of SNRN, the $\lex$-functions are introduced.
Recall that $C_Z (y;)$ denotes $0$.
Every $\lex^k$-function $\f$ indicates which $\lex^k$-predecessor of
$C_w ( \vec y;) = ( C_{w(1)} ( y_1;), \dots, C_{w(k)} ( y_k;))$
should be chosen for each $w \in \Sigma^k$.

\begin{df}[$\lex$-functions]
\normalfont
Suppose that $\f$ is a finite functions such that
$\f : \{1, \dots, k \} \times \Sigma^k \rightarrow \{1, \dots, 2k\}$.
For $w \in \Sigma^k$, let
$C_{w( \f (w))} ( \vec y_{\f (w)};)$ abbreviate 
$(C_{w( \f (1, w))} (y_{\f (1, w))};), \dots, 
                                 C_{w( \f (k, w))} (y_{\f (k, w))};))$.

Then $\f$ is called a $\lex^k$-function, if 
$C_{w( \f (w))} ( \vec y_{\f (w)};) \lex^k C_w ( \vec y; )$
for all $w \in \Sigma^k$
(not depending on choice of $\vec y = (y_1, \dots, y_k)$).
\end{df} 
Let us recall a convention in Notation \ref{convention} that for 
$k_0 \in \{1, \dots, k \}$ and $w \in \Sigma^k$,
$C_{ w( k+k_0 ) } ( y_{ k+k_0 } ; )$ denotes $y_{k_0}$ if 
$w( k_0 ) \in \{ 0, 1 \}$.
By the definition, any $\lex^k$-function $\f$ satisfies the condition:
$\forall w \in \Sigma^k$, 
$\exists k_0 \in \{1, \dots, k\}$
s.t. $w(k_0) \in \{ 0, 1 \}$ and
\begin{equation*}
\label{cases_prec}
 \begin{cases}
 \f (i, w) = i  
 & \text{if} \ i < k_0, \\
 \f (k_0, w) = k+k_0,
 & \text{}  \\
 \f (i, w) \in \{ 1, \dots, 2k\}
 & \text{if} \ i > k_0.
 \end{cases}
\end{equation*}
For examples of $\lex^k$-functions, see \cite{arai_eguchi}.

Now we introduce the class $\N$, which are defined over the natural numbers, contrary to
 the characterization of words in Oitavem \cite{oitavem97}.
Let $a+1$ stand for the numeric successor of $a$, whereas the notation
 $n'$ is used in \cite{oitavem97}.
The notation $a \dot- 1$ is used to denote the corresponding predecessor
 of $a$.

\begin{df}[The class $\N$] 
\label{df_N}
\normalfont
A class $\N^{k, l}$ of functions with $k$ normal and $l$ safe arguments 
is defined by the following initial functions and operations:  
\begin{description}
\item[Zero]
$O^{k, l} \in \N^{k, l}$; \quad
$O^{k, l} (x_1, \dots, x_k; a_1, \dots, a_l) = 0$
\item[Projections]
$I^{k, l}_j \in \N^{k, l}$ $(1 \leq j \leq k+l)$;
  \begin{equation*}
  I^{k, l}_j (x_1, \dots, x_k; a_1, \dots, a_l) =
  \begin{cases}
    x_j
  & \text{if $1 \leq j \leq k$,} \\
    a_{j-k}
  & \text{if $k < j \leq k+l$.}
  \end{cases}
  \end{equation*}
\item[Successor]
$S \in \N^{0, 1}$; \quad
$S (; a) = a+1$
\item[Predecessor]
$P \in \N^{0, 1}$; \quad
$P (; a) = a \dot- 1$
\item[Cases]
$C \in \N^{0, 3}$; \quad
$C (; a, b, c) = 
 \left\{\begin{array}{ll}
         b & \text{if $a=0$} \\
         c & \text{if $a>0$}
        \end{array}
 \right.
$
\item[i-Concatenation]
$C_i \in \N^{1, 0}\ (i=0, 1)$; \quad
$C_i (x; ) = 2x +i =xi$

\item[Deletion]
$D \in \N^{0, 1}$; \quad
$D (; a) = \lfloor a / 2 \rfloor$
\item[Safe composition]
If $h \in \N^{k', l'}$,
$g_1, \dots, g_{k'} \in \N^{k, 0}$ and
$\varphi_1, \dots, \varphi_{l'} \in \N^{k, l}$, then
$f \in \N^{k, l}$ is defined by
$  f ( \vec x; \vec a) =
  h ( \vec g ( \vec x;) ; \vec \varphi ( \vec x; \vec a))
$.

\item[Safe nested recursion on notation (SNRN)]
Suppose that 
$ g \in \mathcal N^{k', l+1} $, and
$ h_w$, 
$ \varphi_w \in \mathcal N^{k+k', l+2} $
for each $w \in \Sigma^k$.
Also suppose that 
$\f_1, \f_2 $
are $\lex^k$-functions.
Then $f \in \mathcal N^{k+k', l+1}$ is defined by 
\begin{equation*}
\left\{
\begin{array}{rcl}
f ( \vec 0, \vec x; \vec a, b) &=& g ( \vec x; \vec a, b), \\
f ( C_w ( \vec y;), \vec x; \vec a, b) &=&
  h_w ( \vec v_1, \vec x; \vec a, b, 
        f ( \vec v_1, \vec x; \vec a, c)) \\ 
&&
[ \varphi_w ( \vec v_2, \vec x; \vec a, b, 
   f ( \vec v_2, \vec x; \vec a, b)) / c] \quad (w \in \Sigma^k), 
\end{array}
\right.
\label{SNRN}
\end{equation*}
where, for every $i=1, 2$, $ \vec v_i$ abbreviates 
$C_{w( \f_i (w))} ( \vec y_{\f_i (w)};)$, 
and hence $ \vec v_i \lex^k C_w ( \vec y; )$.
\end{description}
Then we define
$\N := \bigcup_{k, l \in \mathbb N} \N^{k, l}$
and $\N_{normal} := \bigcup_{k \in \mathbb N} \N^{k, 0}$.
\end{df}

\begin{rem}
\label{rem_snrn}
\normalfont
In the above scheme of SNRN, the restriction of the nesting is not
 crucial.
Even if we allow any constant number of nestings, arguments in
 Section \ref{sec_trs} work.
Further, Arai and the author \cite{arai_eguchi} employ a more general scheme of the form 
\begin{equation}
 \begin{cases}
 f( \vec 0, \vec x; \vec a) = g( \vec x; \vec a), \\
 f( C_w ( \vec y), \vec x; \vec a) =
 h_w ( \vec v_1, \vec x; \vec a,
      f( \vec v_1, \vec x; 
        \vec \varphi_w ( \vec v_2, \vec x; \vec a,
                        f( \vec v_2, \vec x; \vec a)
                       )
       )
     )
 \end{cases}
\label{general_snrn}
\end{equation}
for $ \vec \varphi_w = \varphi_{w, 1}$, $\dots$, $\varphi_{w, l}$ and an $l>0$.
If the projection function $I^{k+k', l}_{k+k'+i}$ is taken as $\varphi_{w, i}$ for each 
$i =1, \dots, l-1$, then this scheme is just the SNRN scheme.
If the scheme (\ref{general_snrn}) is contained instead of the one in
 Definition \ref{df_N},
Theorem \ref{thm_FsubsetFPS} does not hold.
Nevertheless, this restriction is not essential either.
The same class will be generated even by
 (\ref{general_snrn}).
See also Remark \ref{rem_head_reduction}.
\end{rem}

\begin{exam} 
\label{exam_snrn}
\normalfont \
\begin{enumerate}
\item $\text{\d{--}} (x; a) = 
       D^{2^{|x|}} (; a)$.
\begin{eqnarray*}
  \text{\d{--}} (0; a) &=& D (; a), \\
  \text{\d{--}} (xi; a) &=& \text{\d{--}} (x; \text{\d{--}} (x; a)).
\end{eqnarray*}
\label{-}
\item $ + (x; a) = x +a$.
  \begin{eqnarray*}
  +(0; a) &=& a, \\
  +(x0; a) &=& 2x +a 
           = +(x; +(x; a)), \\
  +(x1; a) &=& 2x +1 +a 
           = S(; +(x; +(x; a))).
  \end{eqnarray*}
Similarly, 
$ \minus (x; a) = a \minus x$ is defined.
\label{+}
\item $ \times (x, y; a) = y \cdot 2^{|x|} +a$.
  \begin{eqnarray*}
  \times (0, y; a) &=& y + a, \\
  \times (xi, y; a) &=& \times (x, y; \times (x, y; a)).
  \end{eqnarray*}
\label{times}
\item $f(x; a) = 2^{|x|} +a$.
  \begin{eqnarray*}
  f(0; a) &=& S(;a), \\
  f(xi; a) &=& 2^{|x| +1} +a \\
           &=& f(x; f(x; a)). \quad (i=0, 1)
  \end{eqnarray*}
\label{exp}
\item $f (x, y, z; a) = 2^{|x| \cdot |y| +|z|} +a$.
  \begin{eqnarray*}
  f (0, 0, 0; a) &=& S(;a), \\
  f (x, y, zi; a) &=& 2^{|x| \cdot |y| +|z| +1} +a 
  = f(x, y, z; f(x, y, z; a)), \\
  f (x, yi, 0; a) &=& 2^{|x| (|y| +1)} +a 
  = f(x, y, x; a), \\
  f (xi, 0, 0; a) &=& f(x, 0, 0; a). \quad (i=0, 1)
  \end{eqnarray*}
Similarly, we can define 
$f(x, y, z, u, v, w; a) = 2^{|x| |y| |z| +  |u| |v| + |w|} + a$
and so on.
Hence, a suitable application of safe composition yields
$2^{p(| \vec x|)} \in \N_{normal}$
for any polynomial $p( \vec x)$.
\label{exp3}
\item $2^{2^{|x|}} \not\in \N$, contrary to $ f(x;a) = 2^{2^{|x|}} \cdot
      a \in \mathcal N$ in \cite{arai_eguchi}.
This is due to the choice of initial functions.
The binary successors $S_0, S_1 \in \mathcal N^{0, 1}$ are defined so
      that
$S_i (; a) = 2a + i $ $(i=0, 1)$.
However $C_0, C_1 \in \N^{1, 0}$ are defined on the normal argument, and
      therefore
$C_i$ cannot be $g$ in the scheme of SNRN in Definition \ref{df_N}.
\end{enumerate}
\end{exam}

Recall that $\fps$ is the class of functions computed by a deterministic
Turing machine with the use of a number of cells bounded by $p (|\vec
x|)$ for some polynomial $p ( \vec x)$. 
The following theorem is a direct consequence of results in Section
\ref{sec_trs} and \ref{sec_pspace}. 

\begin{thm}
\label{mainthm}
$\N_{normal} = \fps$.
\end{thm}

\begin{rem} 
\label{rem_uniform}
\normalfont
Let us recall briefly the scheme of \textbf{safe recursion}:
\begin{equation*}
 \begin{cases}
 f(0, \vec x; \vec a) = g( \vec x; \vec a),
 & \text{} \\
 f(S(;y), \vec x; \vec a) = 
 h(y, \vec x; \vec a, f(y, \vec x; \vec a)).
 & \text{}
 \end{cases}
\end{equation*}
Via a work of Bellantoni (\cite{Bel92} Chapter 5), it turns out that if the SNRN
 scheme is replaced by the scheme of safe recursion (and even if both $C_i$
 and $D$ are omitted), the resulting class is 
 identical to the class of linear-space computable functions.
Moreover, replace SNRN by safe nested recursion (on unary
 notation), and restrict safe composition suitably, 
cf. \cite{handley} or \cite{arai_eguchi}.
Then the resulting class will be identical to the class of
 functions computable in $2^{O(|x|)}$-space, i.e., EXPSPACE-computable
 functions.
These observations together with Theorem \ref{mainthm} suggest an uniform
 approach to space-complexity classes.
\end{rem}

\renewcommand{\theequation}{\arabic{equation}}
\setcounter{equation}{0}

\section{Term rewriting system for PSPACE}
\label{sec_trs}

In this section we introduce a term rewriting system $R_{\N}$ over a
class $\mathcal F$ of function symbols corresponding to $\N$.
We show that for any $f \in \mathcal F$, the size of every term
occurring in a rewriting sequence which starts with $f( \vec m; \vec n)$ is bounded
by a polynomial in the binary lengths of numerals $\vec m$ and $\vec n$,
whenever an innermost strategy is used.
This implies that every $R_{\N}$-reduction (innermost) strategy yields an
algorithm for $\N$ running in polynomial space, i.e.,
$ \N_{normal} \subseteq \fps$.

Let $\mathcal V$ be a countably infinite set of variables.
Variables are written as $x, y, z$.
We use $\mathcal V ar (t)$ to denote the set of variables occurring in a
term $t$.
A finite set $R$ of rewrite rules is called a term rewriting system (TRS
for short) if every rewrite rule $l \rightarrow r \in R$ satisfies 
$ \mathcal V ar (r) \subseteq \mathcal V ar (l)$.
A TRS $R$ defines a rewriting relation $\rightarrow_R$ by 
$l \rightarrow r \in R \Rightarrow r[ l \theta ] \rightarrow_R r[ r
\theta ]$
for any context $r[ \cdot ]$ and substitution $\theta$.
Then $l \theta $ is called a redex.
The reflexive and transitive closure of $\rightarrow_R$ is denoted by
$\rightarrow_R^\ast$.

$\mathcal F$ is the smallest class of symbols built up from
$O^{k,l}, I^{k,l}_j, S, P, C, C_i, D $ by means of
$\sub^{k,l}$ and $\snrn^{k,l}$, e.g.,
if $g \in \mathcal F^{k', l+1}$ and 
$h_w, \varphi_w \in \mathcal F^{k+k', l+2}$ for all $w \in \Sigma^k$, then
$\snrn^{k+k', l+1} [g, \{ h_w, \varphi_w : w \in \Sigma^k \}] \in
\mathcal F^{k+k', l+1}$.
Superscripts of $\sub$ and $\snrn$ are always omitted if no
confusion arises.

Let $\mathcal{T(F, V)}$ be the set of terms over $\mathcal F$ and
$\mathcal V$.
Let $\mathcal T (\mathcal F)$ be the set of ground terms in $\mathcal{T(F, V)}$,
which are built up from $O^{0,0}$ by means of elements of $\mathcal F$.
Numerals are built up from $O^{0.0}$ by means of $C_0, C_1$.
We write $n, m, \dots $ for numerals and $0$ instead of $O^{0,0}$.
The binary length $|m|$ of a numeral $m$ is defined by
$|0|=1$, and $C_i (m;) = |m| +1$.
By replacing ``$=$'' occurring in the equations in Definition \ref{df_N}
by ``$\rightarrow$'', we obtain a schematic TRS for $\N$.
However, relaxing the definition, we introduce a non-deterministic or
non-confluent TRS $R_{\N}$
(according to the referee's suggestion).
  
\begin{df}[Term rewriting system $R_{\N}$]
\label{defS} 
\normalfont \
\def\labelitemi{--}
\begin{enumerate}
\item $O^{k,l} ( \vec x; \vec y) \rightarrow 0 \ (0 < k+l)$
\label{SO}
\item $I^{k,l}_j (x_1, \dots, x_k; x_{k+1}, \dots, x_{k+l}) \rightarrow
      x_j \ (1 \leq j \leq k+l)$
\item $S (;0) \rightarrow C_1 (0;)$
\item $S (; C_0 (x;)) \rightarrow C_1 (x;)$
\item $S (; C_1 (x;)) \rightarrow C_0 ( S (;x);)$
\item $P (;0) \rightarrow 0$
\item $P (; C_0 (x;)) \rightarrow C_1 ( P (;x);)$
\item $P (; C_1 (x;)) \rightarrow C_0 (x;)$
\item $C (; 0, x, y) \rightarrow x$
\item $C (; C_i (z;), x, y) \rightarrow y$
\item $D (;0) \rightarrow 0$
\item $D (;C_i (x;)) \rightarrow x$
\label{SD}
\item $\sub [h, \vec g, \vec \varphi ] ( \vec x; \vec y) \rightarrow h (
      \vec g ( \vec x;); \vec \varphi ( \vec x; \vec y))$
\label{Ssub}
\item $\snrn [g, \{ h_w, \varphi_w : w \in \Sigma^k \}] ( \vec 0, \vec
      x; \vec z, z') \rightarrow g ( \vec x; \vec  z, z')$
\label{Ssnrn1}
\item $\snrn [g, \{ h_w, \varphi_w : w \in \Sigma^k \} ] 
       (C_{w'} ( \vec y;), \vec x; \vec z, z') \rightarrow \\ \quad
       h_{w'} ( \vec v_1, \vec x; \vec z, z',
             \snrn [g, \{ h_w, \varphi_w : w \in \Sigma^k \} ] 
             ( \vec v_1, \vec x; \vec z, u)) \\ \quad
       [ \varphi_{w'} ( \vec v_2, \vec x; \vec z, z', 
                     \snrn [g, \{ h_w, \varphi_w : w \in \Sigma^k \} ] 
                     ( \vec v_2, \vec x; \vec z, z')) / u] 
      $ \\
for some $\vec v_1, \vec v_2 \lex^k C_{w'} ( \vec y; )$ $(w' \in \Sigma^k)$.
\label{Ssnrn2}
\end{enumerate}
\end{df}


The intended semantic $\pi$ for $\mathcal F$ is obvious.
For example, suppose that 
$ f = \sub [h, \vec g, \vec \varphi ]$. 
Then $\pi (f) \in \N$ is defined by safe composition from  
$\pi (h)$, $\pi ( \vec g)$ and $\pi (\vec \varphi)$ in $\N$:
$\pi (f) ( \vec x; \vec a) =
 \pi (h) ( \pi ( \vec g) ( \vec x; ) ; 
           \pi ( \vec \varphi ) ( \vec x; \vec a)
         )$.
However, we should be careful of the case $f= \snrn [g, \{ h_w, \varphi_w : w \in \Sigma^k \}]$.
Clearly, the TRS $R_{\N}$ is not confluent due to absence of $\lex^k$-functions
 $\f_1, \f_2$.
Let us recall that a $\lex$-function indicates which $\lex$-predecessor
 should be chosen.
The mapping $\pi$ is extended to $\mathcal{T(F)} \rightarrow \mathbb N$ by
$\pi (f( \vec t; \vec s)) := \pi (f) ( \pi (\vec t); \pi (\vec s))$.
\begin{lem}
\label{lem_termination}
$R_{\N}$ is terminating.
\end{lem}

\begin{proof}
Define a precedence $<_{\mathcal F}$ on $\mathcal F$ by
\def\labelitemi{--}
\begin{itemize}
\item $0 <_{\mathcal F} O^{k, l}$ if $k+l >0$,
\item $g <_{\mathcal F} f$ for each
      $g \in \{ O^{k,l}, I^{k,l}_j, C_i \}$ and 
      $f \in \{ S, P, C, D \}$, 
\item $f <_{\mathcal F} \sub [h, \vec g, \vec \varphi ]$ for each
      $f \in \{ h, \vec g, \vec \varphi \}$, and
\item $f <_{\mathcal F} 
       \snrn [g, \{ h_w, \varphi_w : w \in \Sigma^k \} ]
      $ for each 
      $f \in \{ g, h_w, \varphi_w : w \in \Sigma^k \}$.
\end{itemize}
Then $R_{\N}$ is reducing under the lexicographic path order LPO induced by
$<_{\mathcal F}$.
Namely, $t \RN s \Rightarrow s <_{\mathrm{LPO}} t$ for each 
$t, s \in \mathcal{T(F, V)}$.
The well-foundedness of LPO implies the termination of $R_{\N}$.
\end{proof}

By the termination of $R_{\N}$, a normal form always exists, i.e.,
for any $t \in \mathcal{T(F, V)}$,  there exists a $t' \in \mathcal{T(F,
V)}$
such that $t \RN^\ast t'$ and $\nexists s$ s.t. $t' \RN s$.
It is not difficult to see that any normal form of each ground term is a
numeral. 

\begin{notation} \normalfont
Since $R_{\N}$ is not confluent,
there may be several normal forms of a ground term.
At the risk of confusion, let $\ul{t}$ denote a normal form of $t \in \mathcal{T(F)}$.
Hence, in paticular, $\ul{m} = m$ for any numeral $m$.
\end{notation}

\begin{df} \normalfont
The length $\lh (f)$ of $f \in \mathcal F$ is defined by
 \def\labelitemi{--}
 \begin{itemize}
 \item $\lh (f) =1$ if $f \in \{ O^{k,l}, I^{k,l}_j, S, P, C, C_i, D
                              \}$,
 \item $\lh (\sub [h, \vec g, \vec \varphi ]) =
        \lh (h) + \sum_{i=1}^{k'} \lh (g_i) +
        \sum_{i=1}^{l'} \lh (\varphi_i) + 1,
       $ and
 \item $\lh (\snrn [g, \{ h_w, \varphi_w  : w \in \Sigma^k \}])$
       $=\lh (g) + \sum_{w \in \Sigma^k} \lh (h_w)$ $+
        \sum_{w \in \Sigma^k} \lh (\varphi_w)$ $+1.
       $
 \end{itemize}
Then, the length $\lh (t)$ of  
$t \in \mathcal T (\mathcal F)$ is defined by
\[
 \lh (f( \vec t; \vec s)) = \lh (f) + \sum_{i=1}^k \lh (t_i) +
                            \sum_{i=1}^l \lh (s_i).
\]
\end{df}
Hence, in particular, $|m| = \lh (m) $ for any  numeral $m$. 
As a corollary, we observe Lemma \ref{lem_lh(t)}.
\begin{lem}
\label{lem_lh(t)}
If $t \rightarrow s \in R_{\N}$ by any of Definition
$\ref{defS}.\ref{SO}$ -- $\ref{defS}.\ref{SD}$, then
$\lh (s \theta) \leq \lh (t \theta)$ for any ground substitution $\theta$.
\end{lem}

In general, for $f \in \mathcal F$, 
$\max \{ \lh (t) : f(\vec m; \vec n) \RN^\ast t \}$ is not bounded by a
polynomial in $| \vec m |, | \vec n |$.
Consider the rewriting rule 
$f( C_i (y;); x) \rightarrow f(y; f(y; x))$.
For numerals $m$ and $n$,
it is possible to obtain 
the rewriting sequence
$f( m; n) \RN^\ast 
 f(0; f(0; \cdots f(0; n) \cdots ))$.
However, 
$\lh ( f(0; f(0; \cdots f(0; n) \cdots )) )=
 ( \lh (f) +1) \cdot 2^{|m|}+ |n|$,
which is not admissible.
Hence we employ the innermost rewriting strategy.

Let us use $t \rightarrow_{in} s$ to denote $t \RN s$ according to an
innermost strategy.
Namely, $t \rightarrow_{in} s$ if $t \RN s$ and the redex $r$ is
innermost, i.e., $r$ contains no redex that is a proper subterm of $r$.
The rest of this section is devoted to show that for every 
$f \in \mathcal F$ there exists a polynomial $p$ such that
$\max \{ \lh (t) : f ( \vec m; \vec n) \RS^\ast t \} \leq p (| \vec m|, |\vec n|)$.
Due to the innermost strategy,
$\max \{ \lh (t) : f( \vec t; \vec s) \RN^\ast t \}$ can be reduced as follows.

\begin{lem}
\label{lem_sp} 
For any $f \in \mathcal F^{k, l}$ and 
$t_1, \dots, t_k, s_1, \dots, s_l \in \mathcal T (\mathcal F)$,
$\max \{ \lh (t) : f( \vec t; \vec s) \RS^\ast t \}$ is bounded by either 
$\lh (f) + \sum_{i=1}^k \max \{ \lh (t_i^\ast) : t_i \RS^\ast t_i^\ast \} 
         + \sum_{i=1}^l \max \{ \lh (s_i^\ast) : s_i \RS^\ast s_i^\ast \}
$, or
$\max \{ \lh (t) : f( \ul{\vec t}; \ul{\vec s} ) \RS^\ast t \}$.
\end{lem}

In the following proofs, we use $\Sp (t)$ to abbreviate 
$\max \{ \lh (s): t \RS^{\ast} s \}$,
which intended to denote the \emph{space} 
required to rewrite $t \in \mathcal T (\mathcal F)$ according to 
$\rightarrow_{in}$.  

\begin{proof}
Suppose that 
$f \in \{ O^{k, l}, I^{k, l}_j, C_i, D, S, P, C \}$.
From Lemma \ref{lem_lh(t)}, we observe that 
\begin{equation}
\label{initialcase}
 \Sp (f( \vec t; \vec s)) \leq
 \lh (f) + \sum_{i=1}^k \Sp (t_i) + \sum_{i=1}^l \Sp (s_i).
\end{equation}
Otherwise, 
$f \in \{ \sub [h, \vec g, \vec \varphi ],
          \snrn [g, \{ h_w, \varphi_w : w \in \Sigma^k \}]
       \}
$.
Then, due to the innermost strategy,
an arbitrary terminating rewriting which starts with 
$f( \vec t; \vec s)$ runs as 
\begin{eqnarray*}
f( \vec t; \vec s) &\RS^\ast&
f( \ul{\vec t} ; \ul{ \vec s}) 
\RS^\ast \ul{f( \ul{\vec t}; \ul{\vec s})}.
\end{eqnarray*}
Thus, $\Sp (f( \vec t; \vec s))$ is bounded by either 
$\lh (f) + \sum_{i=1}^k \Sp (t_i) + \sum_{i=1}^l \Sp (s_i)$ or
$\Sp (f( \ul{\vec t} ; \ul{ \vec s}))$.
This concludes the lemma.
\end{proof}
 
\begin{df}
\normalfont
For $d, b_1, \dots, b_k \in \mathbb N$, put
\[
 \sum (d, \vec b) =
 \sum (d, b_1, \dots, b_k):=
 \sum_{i=1}^k ( \max \vec b +1 )^{d-i} b_i.
\]
\end{df}
Then, as in \cite{arai_eguchi}, we can prove a fundamental lemma concerning
the descending lengths with respect to the ordering $\lex$. 

\begin{lem} 
\label{lem_pred}
Suppose that $\vec n$ and $\vec m$ are numerals.
If $ \vec n \lex^k \vec m$ and $k \leq d$, then
$\sum (d, | \vec n| ) < \sum (d , | \vec m| )$.
\end{lem}

\begin{lem}
\label{lem_lh}
For any $f \in \mathcal F$
there exist some constants $c$ and $d$ such that
for numerals $\vec m, \vec n$ and an arbitrary normal form 
$\ul{ f( \vec m; \vec n)}$ of $f( \vec m; \vec n)$,
$ 2^{| \ul{f( \vec m; \vec n)}|} \leq
 2^{c( \max |m|^d + 1)} +
 2^{ \max |n|}
$,
and hence
$| \underline{f ( \vec m; \vec n)}| \leq
\max \{ c( \max |m|^d +1 ),
        \max |n|
     \} +1.
$
\end{lem}

\begin{proof}
We prove the lemma by induction on $\lh (f)$.
For the base case, if $\lh (f) =1$, i.e., 
$f \in \{O, I, C_i, S, P, D, C\}$, then
$2^{| \ul{f( \vec m; \vec n)}|} \leq 
 2^{\max | \vec m| +1} + 2^{\max | \vec n|}
$.

Consider the induction step.
Suppose that $f = \sub [h, \vec g, \vec \varphi ]$ for
$h \in \N^{k', l'}$, 
$g_1, \dots, g_{k'} \in \N^{k, 0}$ and
$\varphi_1, \dots, \varphi_{l'} \in \N^{k, l}$. 
Then, by the induction hypothesis, there exist some constants 
$c_0, d_0$ for $h$, 
$c_i, d_i$ for $g_i$ $(i=1, \dots, k')$ and 
$c_{k+i}, d_{k+i}$ for $\varphi_{i}$ $(i=1, \dots, l')$
which enjoy the condition.
Fix an normal form $\ul{g_i ( \vec m;)}$ of 
$g_i ( \vec m;)$ for each $i=1, \dots, k'$ and
$\ul{\varphi_i ( \vec m; \vec n)}$ of 
$\varphi_i ( \vec m; \vec n)$ for each $i=1, \dots, l'$.
Let
$c' := \max \{ c_0, c_1, \dots, c_{k'+l'} \}$ and
$d' := \max \{ d_0, d_1, \dots, d_{k'+l'} \}$.
Then, by IH for $\vec g$ and $\vec \varphi$,
\begin{eqnarray}
|\ul{g_i ( \vec m;)}| &\leq&
c' ( \max |\vec m|^{d'} +1) \
(i=1, \dots, k'), 
\label{lh_g} \\
2^{|\ul{\varphi_i ( \vec m; \vec n)}|} &\leq&
2^{ c' ( \max |\vec m|^{d'} +1)
  } +
2^{ \max |\vec n|} \
(i=1, \dots, l').
\label{lh_phi}
\end{eqnarray}
Fix an normal form 
$\ul{f( \vec m; \vec n)} = 
 \ul{h( \ul{ \vec g( \vec m;)};
        \ul{ \vec \varphi ( \vec m; \vec n)}
      )
    }
$
of $f( \vec m; \vec n)$.
By IH for $h$,
\begin{eqnarray}
2^{|\ul{f( \vec m; \vec n)}|} &=&
2^{|\ul{h( \ul{\vec g ( \vec m; \vec n)}; 
           \ul{\vec \varphi ( \vec m; \vec n)}
        )
       }|
  } 
\nonumber 
\\
&\leq&
2^{ c' ({ \max |\ul{ \vec g( \vec m;)}|}^{d'} +1)
  } +
2^{ \max |\ul{ \vec \varphi ( \vec m; \vec n)}|
  }.
\label{lh_h}
\end{eqnarray}
Put $c:= 2^{d'} c'^{d'+1} +1$ and 
$d := d'^2$.
Then,
\begin{eqnarray}
c' ({ \max |\ul{ \vec g( \vec m;)}|}^{d'} +1) &\leq&
c' ({ c' ( \max |\vec m|^{d'} +1)}^{d'} +1)
\quad \text{by (\ref{lh_g})}, 
\nonumber \\ 
&\leq&
c' ( c'^{d'} 2^{d'} \max |\vec m|^{d'^2} +1) 
\nonumber
\\ &\leq&
2^{d'} c'^{d'+1} ( \max |\vec m|^{d'^2} +1)
\nonumber
\\ &\leq&
(c-1) ( \max |\vec m|^d +1)
\quad \text{by Def. of } c, d.
\label{lh_h(g)}
\end{eqnarray}
By (\ref{lh_h}), (\ref{lh_h(g)}) and (\ref{lh_phi}),
\begin{eqnarray*}
&&
2^{|\ul{f( \vec m; \vec n)}|} \\
&\leq&
2^{(c-1) ( \max |\vec m|^d +1)} +
2^{ c' ( \max |\vec m|^{d'} +1)
  } +
2^{ \max |\vec n|} 
\\ &\leq&
2^{(c-1) ( \max |\vec m|^d +1) +1} +
2^{ \max |\vec n|} 
\quad \text{by } c' \leq c-1,
\\ &\leq&
2^{c ( \max |\vec m|^d +1) } +
2^{ \max |\vec n|}.
\end{eqnarray*}
 
For the case for SNRN, it is convenient to consider a more general form
(\ref{general_snrn}) in Remark \ref{rem_snrn}:
\begin{equation*}
 \begin{cases}
 f( \vec 0, \vec x; \vec a) = g( \vec x; \vec a), \\
 f( C_w ( \vec y), \vec x; \vec a) =
 h_w ( \vec v_1, \vec x; \vec a,
      f( \vec v_1, \vec x; 
        \vec \varphi_w ( \vec v_2, \vec x; \vec a,
                        f( \vec v_2, \vec x; \vec a)
                       )
       )
     ) \\
 (\vec \varphi_w = \varphi_{w, 1}$, $\dots$, $\varphi_{w, l}, \ l>0)
 \end{cases}
\end{equation*}
Let
$f = \snrn [g, \{ h_w, \vec \varphi_w : w \in \Sigma^k \}] 
 \in \mathcal F^{k+k', l}
$
for
$g \in \mathcal F^{k', l}$ and
$h_w$, $\varphi_{w, 1}$, $\dots$, $\varphi_{w, l}$ 
$\in \mathcal F^{k+k', l+1}$. 
By IH, there exist some constants $c_0, d_0$ for $g$,
$c_1, d_1$ for $h_w$, and 
$c_2, d_2$ for $\vec \varphi_w$ which enjoy the condition for all 
$w \in \Sigma^k$.
Put $c:= \max \{ c_0, c_1, c_2, 2  \}$ and
$d:= \max \{ d_0, d_1, d_2, k \}$.
Suppose that
$m^y_1, \dots, m^y_k$ and $m^x_1, \dots, m^x_{k'}$ are numerals.
Then, by side induction on 
$\sum (d, | \vec m^y | )$,
we show that for any
$\vec t = t_1, \dots, t_l \in \mathcal{T(F)}$ and 
their arbitrary normal forms 
$\ul{\vec t} = \ul{t_1}, \dots, \ul{t_l}$,
\[
 2^{| \ul{f( \vec m^y, \vec m^x; \ul{\vec t})}|} \leq
 2^{c (( \max \{ |\vec m^y|, |\vec m^x| \} +1)^d + 
       \sum (d, | \vec m^y | ) +1
      )
   } +
 2^{\max | \ul{\vec t}|}.
\]
Put 
$b := \max \{ |\vec m^y|, |\vec m^x| \} +1$ and
$b_i := |m^y_i|$ 
for each $i=1, \dots, k$.
First consider the case $\vec m^y = \vec 0$:
$f( \vec 0, \vec m^x; \ul{\vec t}) \RS g( \vec m^x; \ul{\vec t})$.
This case follows from IH as
\begin{eqnarray*}
2^{| \ul{f( \vec 0, \vec m^x; \ul{\vec t})}|}
&=&
2^{| \ul{g( \vec m^x; \ul{\vec t})}|} \\
&\leq&
2^{c_0 ( b^{d_0} +1 )} + 2^{\max | \ul{\vec t}|} 
\quad \text{by IH for } g, \\
&\leq&
 2^{c ( b^d + \sum (d, \vec b) +1 )} +
 2^{\max | \ul{\vec t}|}
\quad \text{by } c_0 \leq c, \text{ and } d_0 \leq d.
\end{eqnarray*}
Next consider the case $\vec m^y \neq \vec 0$:
\begin{eqnarray*}
f( \vec m^y, \vec m^x; \ul{\vec t}) &\RS&
 h_w ( \vec m^{v}_1, \vec m^x; \ul{\vec t},
      f( \vec m^{v}_1, \vec m^x; 
        \vec \varphi_w ( \vec m^{v}_2, \vec m^x; \ul{\vec t},
                        f( \vec m^{v}_2, \vec m^x; \ul{\vec t})
                       )
       )
     )
\\
&\RS^\ast&
 h_w ( \vec m^{v}_1, \vec m^x; \ul{\vec t},
      f( \vec m^{v}_1, \vec m^x; 
        \vec \varphi_w ( \vec m^{v}_2, \vec m^x; \ul{\vec t},
                        \ul{f( \vec m^{v}_2, \vec m^x; \ul{\vec t})}
                       )
       )
     )
\\
&\RS^\ast&
 h_w ( \vec m^{v}_1, \vec m^x; \ul{\vec t},
      f( \vec m^{v}_1, \vec m^x; 
\ul{        \vec \varphi_w ( \vec m^{v}_2, \vec m^x; \ul{\vec t},
                        \ul{f( \vec m^{v}_2, \vec m^x; \ul{\vec t})}
                       )
}       
       )
     )
\\
&\RS^\ast&
 h_w ( \vec m^{v}_1, \vec m^x; \ul{\vec t},
\ul{      f( \vec m^{v}_1, \vec m^x; 
\ul{        \vec \varphi_w ( \vec m^{v}_2, \vec m^x; \ul{\vec t},
                        \ul{f( \vec m^{v}_2, \vec m^x; \ul{\vec t})}
                       )
}       
       )
}     ).
\end{eqnarray*}
for some
$ \vec m^{v}_1, \vec m^{v}_2 \lex^k \vec m^y$ and 
a suitable $w \in \Sigma^k$.

By Lemma \ref{lem_pred}, for each $j=1, 2$,
\begin{equation}
 \sum (d, |\vec m^v_j| ) < \sum (d, \vec b).
\label{sih}
\end{equation}
Hence, by the side induction hypothesis and 
$\max | \vec m^v_2 | \leq \max | \vec m^y |$,
\begin{equation*}
 2^{| \ul{f( \vec m^v_2, \vec m^x; \ul{\vec t} )}|} \leq
 2^{c ( b^d + \sum (d, | \vec m^v_2 | ) +1)} +
 2^{\max | \ul{\vec t}|} \leq
 2^{c ( b^d + \sum (d, \vec b) )} +
 2^{\max | \ul{\vec t} |}.
\end{equation*}
From this and IH for $\vec \varphi_w$,
\begin{eqnarray}
&&
2^{\max |
\ul{        \vec \varphi_w ( \vec m^{v}_2, \vec m^x; \ul{\vec t},
                        \ul{f( \vec m^{v}_2, \vec m^x; \ul{\vec t})}
                       )
}|       
} 
\nonumber
\\
&\leq&
2^{c_2 ( b^{d_2} +1)} +
2^{ \max \{ |\ul{\vec t}|, | \ul{f( \vec m^v_2, \vec m^x; \ul{\vec t} )}| \}
  } 
\nonumber
\\
&\leq&
2^{c_2 ( b^{d_2} +1)} +
 2^{c ( b^d + \sum (d, \vec b) )} +
 2^{\max | \ul{\vec t} |} 
\nonumber
\\
&\leq&
 2 \cdot
 2^{c ( b^d + \sum (d, \vec b) )} +
 2^{\max | \ul{\vec t} |} 
\quad \text{by } c_2 \leq c, \ d_2 \leq d \text{ and } 
1 \leq \sum (d, \vec b).
\label{sih_2}
\end{eqnarray}
Similarly to the case for $\vec m^v_2$, 
by SIH together with (\ref{sih}) for $\vec m^v_1$,
\begin{eqnarray*}
&&
2^{|
\ul{      f( \vec m^{v}_1, \vec m^x; 
\ul{        \vec \varphi_w ( \vec m^{v}_2, \vec m^x; \ul{\vec t},
                        \ul{f( \vec m^{v}_2, \vec m^x; \ul{\vec t})}
                       )
}       
       )
}|
} \\
&\leq&
 2^{c ( b^d + \sum (d, | \vec m^v_1 | ) +1)} +
2^{\max |
\ul{        \vec \varphi_w ( \vec m^{v}_2, \vec m^x; \ul{\vec t},
                        \ul{f( \vec m^{v}_2, \vec m^x; \ul{\vec t})}
                       )
}|       
} 
\\
&\leq&
 2^{c ( b^d + \sum (d, \vec b ))} +
 2 \cdot
 2^{c ( b^d + \sum (d, \vec b) )} +
 2^{\max | \ul{\vec t} |} 
\quad \text{by (\ref{sih_2}),} \\
&\leq&
 3 \cdot
 2^{c ( b^d + \sum (d, \vec b) )} +
 2^{\max | \ul{\vec t} |}. 
\end{eqnarray*}
Therefore,
\begin{eqnarray*}
&&
 2^{| \ul{f( \vec m^y, \vec m^x; \ul{\vec t})}|} \\
&=&
2^{\ul{|
 h_w ( \vec m^{v}_1, \vec m^x; \ul{\vec t},
\ul{      f( \vec m^{v}_1, \vec m^x; 
\ul{        \vec \varphi_w ( \vec m^{v}_2, \vec m^x; \ul{\vec t},
                        \ul{f( \vec m^{v}_2, \vec m^x; \ul{\vec t})}
                       )
}       
       )
}     )
}|
} \\
&\leq&
2^{c_1 (b^{d_1} +1)} +
3 \cdot 2^{c ( b^d + \sum (d, \vec b))} +
2^{\max | \ul{\vec t}|} 
\quad \text{by IH for } h_w, 
\\
&\leq&
4 \cdot 2^{c ( b^d + \sum (d, \vec b))} +
2^{\max | \ul{\vec t}|} 
\quad \text{by} \ c_1 \leq c, \ d_1 \leq d \ \text{and} \ 
                1 \leq \sum (d, \vec b), \\
&\leq&
2^{c ( b^d + \sum (d, \vec b) +1)} +
2^{\max | \ul{\vec t}|} \quad \text{by} \ 2 \leq c. 
\end{eqnarray*}
This completes the proof of the lemma.
\end{proof}


\begin{thm}
\label{thm_FsubsetFPS}
For any $f \in \mathcal F^{k,l}$, there exist $c$ and $d$ such that
for numerals $m_1, \dots, m_k$ and 
$n_1, \dots, n_l$,
$ \max \{ \lh (t) : f( \vec m; \vec n) \RS^\ast t \} \leq
 c ( \max | \vec m|^d +1) ( \max |\vec n| +1)
$.
\end{thm}

\begin{proof}
We prove the theorem again by induction on $\lh (f)$. 
If $\lh (f) = 1$, then, by (\ref{initialcase}), 
\begin{eqnarray}
\Sp (f ( \vec m; \vec n)) 
&\leq&
\lh (f) + \sum_{i=1}^k |m_i| + \sum_{i=1}^l |n_i|
\nonumber \\
&\leq&
( \lh (f) +k +l) ( \max |\vec m| +1) ( \max | \vec n| +1).
\label{lh=1}
\end{eqnarray}

The case  
$f = \sub [h, \vec g, \vec \varphi]$
is seen from IH, Lemma \ref{lem_sp}, and Lemma \ref{lem_lh}.

Finally, consider the case 
$f = \snrn [g, \{ h_w, \varphi_w : w \in \Sigma^k \}]$. 
For simplicity, we only consider $f \in \mathcal F^{k, 1}$ such that
$f( \vec 0; x) \rightarrow g(; x)$ and 
$f( \vec y; x) \rightarrow f( \vec v_1; f( \vec v_2; x))$
for some $ \vec v_1, \vec v_2 \lex^k \vec y$.
By IH, there exists a constant $c_g$ for $g$ 
enjoying the condition.
Furthermore, by Lemma \ref{lem_lh}, there exist constants 
$c', d'$ for $f$ enjoying the condition in the lemma.
Let $c_\lh := \lh (f) + k+1$.
Now put 
$c := \max \{ c_\lh, c_g, c'+1 \}$ and 
$d := \max \{ k, d' \}$.
Suppose that 
$m_1, \dots, m_k$ are numerals.
Then, by side induction on $\sum (d, |\vec m|)$,
we show that for any $t \in \mathcal{T(V)}$ and its arbitrary normal
 form $\ul{t}$,
\begin{equation}
 2^{\Sp (f( \vec m; \ul{t})
        )
   }
 \leq
 ( 2^{c ( \sum (d, |\vec m| ) +1)} +
   2^{ | \ul{t}| +1
     }
 )^{c ( \sum (d, |\vec m| ) +1)
   }.
\label{sp_f}
\end{equation}
This results in 
$ \Sp (f ( \vec m; \ul{t})) \leq
  c^2 ( \sum (d, |\vec m| ) +1
       )^2
 ( |\ul{t}| +1
 )
$.
\\
Case 1:
$\Sp ( f( \vec m; \ul{t})) \leq
 \lh (f) + \sum_{i=1}^k |m_i| + |\ul{t}|
$;
In this case, (\ref{sp_f}) follows from
\begin{eqnarray}
\Sp (f ( \vec m; \ul{t})) 
&\leq&
\lh (f) + k \max |\vec m| + |\ul{t}| 
\nonumber \\
&\leq&
c_\lh ( \max |\vec m| + |\ul{t}| +1) 
\nonumber \\
&\leq&
c ( \max |\vec m| +1) ( |\ul{t}| +1)
\quad \text{by} \ c_\lh \leq c.
\nonumber
\end{eqnarray}
Case 2:
$\Sp ( f( \vec m; \ul{t})) \leq
 \Sp (g( ;  \ul{t}))
$;
In this case, (\ref{sp_f}) follows from
\begin{eqnarray*}
\Sp (g( ; \ul{t})) 
&\leq&
c_g ( |\ul{t}| +1) 
\quad \text{by IH for} \ g, \\
&\leq&
c ( \max |\vec m|^d +1) ( |\ul{t}| +1) \quad \text{by} \
c_g \leq c.
\end{eqnarray*}
Case 3: Otherwise;
By Lemma \ref{lem_sp}, 
\begin{eqnarray}
\Sp (f( \vec m; \ul{t}))
&\leq&
\Sp (f( \vec m_1; \ul{f( \vec m_2; \ul{t})}))
\label{sp_f3}
\end{eqnarray}
for some $\vec m_1, \vec m_2 \lex^k \vec m$ and a normal form 
$\ul{f( \vec m_2; \ul{t})}$ of $f( \vec m_2; \ul{t})$.
Again, by Lemma \ref{lem_pred}, for each $j=1, 2$,
\begin{equation}
\sum (d, |\vec m_j|) < \sum (d, |\vec m|).
\label{sih'}
\end{equation}
Hence, by SIH for $\vec m_1$,
\begin{eqnarray}
2^{ \Sp (f( \vec m_1; \ul{f( \vec m_2; \ul{t})}))
  }
&\leq&
 ( 2^{c ( \sum (d, |\vec m_1| ) +1)} +
   2^{ | \ul{f( \vec m_2; \ul{t})}| +1
     }
 )^{c ( \sum (d, |\vec m_1| ) +1)
   }
\nonumber
\\
&\leq&
 ( 2^{c \sum (d, |\vec m| ) } +
   2^{ | \ul{f( \vec m_2; \ul{t})}| +1
     }
 )^{c ( \sum (d, |\vec m| ) +1 ) 
   }
\quad \text{by } (\ref{sih'}).
\label{Sp(f_1)}
\end{eqnarray}
On the other hand, by Lemma \ref{lem_lh}, 
\begin{eqnarray}
2^{|\ul{f( \vec m_2; \ul{t})}| +1} &\leq&
2^{c' ( \max |\vec m_2|^{d'} +1) +1} +
2^{|\ul{t}| +1}
\nonumber
\\
&\leq&
2^{c ( \max |\vec m_2|^d +1)} +
2^{|\ul{t}| +1}
\quad \text{by } c'+1 \leq c \text{ and } d' \leq d,
\nonumber
\\
&\leq&
2^{c \sum (d, |\vec m|) } +
2^{| \ul{t} | +1}
\quad \text{by } \max |\vec m_2|^d +1 \leq \sum (d, |\vec m|).
\label{lh(f_2)'} 
\end{eqnarray}
Combining (\ref{Sp(f_1)}) and (\ref{lh(f_2)'}), we obtain
\begin{eqnarray}
2^{ \Sp (f( \vec m_1; \ul{f( \vec m_2; \ul{t})}))
  }
&\leq&
 ( 2^{c \sum (d, |\vec m| ) } +
2^{c \sum (d, |\vec m|) } +
2^{| \ul{t} | +1}
 )^{c (\sum (d, |\vec m| ) +1) 
   }
\nonumber
\\
&=&
 ( 2^{c \sum (d, |\vec m| ) +1} +
2^{| \ul{t} | +1}
 )^{c ( \sum (d, |\vec m| ) +1) 
   }
\nonumber
\\
&\leq&
 ( 2^{c (\sum (d, |\vec m| ) +1)} +
2^{| \ul{t} | +1}
 )^{c ( \sum (d, |\vec m| ) +1) 
   }
\quad \text{by } 1 \leq c.
\label{Sp(f_1)'}
\end{eqnarray}
Now (\ref{sp_f}) follows from (\ref{sp_f3}) and (\ref{Sp(f_1)'}).
Above argument is slightly extended to the case for the general form of
 SNRN.
This completes the case for SNRN, and, therefore, the proof of the theorem.
\end{proof}

\begin{rem}
\label{rem_head_reduction}
\normalfont
As mentioned in Remark \ref{rem_snrn}, if a more general form
 (\ref{general_snrn}) of SNRN is taken instead,
the proof of Theorem \ref{thm_FsubsetFPS} does not work.
However, with the use of a stronger rewriting strategy, e.g. head
 reduction, a similar argument will work.
\end{rem}

\begin{col}
\label{thm_NsubsetFPS}
If $f \in \N_{normal}$, then the space required to compute 
$f( \vec x;) $ is bounded by 
$p ( |\vec x|)$ for some polynomial $p$, i.e.,
$ \N_{normal} \subseteq \fps$.
\end{col}

\begin{proof}
Assume that $f \in \N_{normal} $ and $\vec m \in \mathbb N$.
Let us identify $f$ with the corresponding function symbol in 
$\mathcal F$ and $\vec m$ with the corresponding numerals in 
$\mathcal{T(F)}$.
Recall that $\Sp (m) = |m|$ for any numeral $m$.
This together with Theorem \ref{thm_FsubsetFPS} yields a polynomial $p$
 such that 
$\Sp (f ( \vec m; )) \leq p (| \vec m|)$.
From this, there exists an algorithm that rewrites an arbitrary ground term into its
 reduct running in polynomial time and hence in polynomial space.
And also, any normal form $\ul{f ( \vec m; )}$ of $f( \vec m; )$ is a numeral
and there is a normal form $\ul{f ( \vec m; )}$ such that 
$\pi (f( \vec m; )) = \pi ( \ul{f ( \vec m; )})$. 
Therefore, we can construct a non-deterministic Turing program which computes $f$ in polynomial
 space.
Recall here Savitch's Theorem which states 
$\textbf{NPSPACE} \subseteq \textbf{PSPACE}$. 
This yields $ \textbf{FNPS} \subseteq \fps$, and hence $f \in \fps$.
\end{proof}

\section{PSPACE-computable functions belong to $ \N$}
\label{sec_pspace}

In this section 
we show that every function from $\fps$ is a member of $\N_{normal}$.
The proof is divided into two steps.

First we simulate a computation of a $O(|x|^k)$-space-bounded Turing
machine by an $2^{O(|x|^l)}$-step-bounded register machine working over
the unary representation (for some $l$).
This is done by imitating an argument in 
Handley and Wainer \cite{handley}, in which 
every $O(|x|)$-space Turing computation is simulated by a 
$2^{O(|x|)} $($\approx O(x^k)$ for some $k$)-step-bounded register machine on unary notation. 

Then, we simulate an action of the exponentially bounded register
machine by functions in $\N_{normal}$. 
The argument is similar to one in Arai and the author
\cite{arai_eguchi}, in which an exponential-time Turing computation is
arithmetized with the use of SNRN.
We notice that the situation is, however, easier.
This is due to the fact that the class \textbf{PSPACE} is closed under
composition whereas \textbf{EXP} is not.

To simulate computations within $\N$, 
the closedness under a simultaneous SNRN is necessary:  

\begin{lem} \label{lem_simul}
{\normalfont (Cf. \cite{arai_eguchi} Theorem 3.1)}
Suppose that $f_1, \dots, f_l$ are defined from
$h_1, \dots, h_l \in \N^{k', l}$ and $\lex^k$-functions $\f_1, \f_2$ by a scheme of simultaneous SNRN such that
for each $i=1, \dots, l$ and $w \in \Sigma^k$,
\begin{equation*}
 \begin{cases}
 f_i ( \vec 0, \vec x; \vec a) = h_i ( \vec x; \vec a), 
 & \text{} \\
 f_i ( C_w ( \vec y;), \vec x; \vec a) =
 f_i ( \vec v_1, \vec x;
       f_1 ( \vec v_2, \vec x; \vec a), \dots, 
       f_l ( \vec v_2, \vec x; \vec a)
     ).
 & \text{}
 \end{cases}
\end{equation*}
Then, for any $g_1, \dots, g_l \in \N^{k',0}$ and for each 
$i=1, \dots, l$,
\[
 f_i ( \vec y, \vec x; g_1 ( \vec x; ), \dots,
                       g_l ( \vec x; )
     ) \in \N^{k+k', 0}.
\]
\end{lem}

\begin{proof}
Although the usual simultaneous recursion is easily
 reduced to a single recursion by a pairing and unpairing
 functions,
it is not a trivial task in the safe representation. 
Bellantoni \cite{Bel92} defines a pairing function in the Bellantoni-Cook
 class $\mathcal B$.
This pairing function is not definable in $\N$, since the
 i-concatenation $C_i$ is available only on the normal argument.
A familiar pairing function
$\langle x, y \rangle = \frac{1}{2} (x+y) (x+y+1) +x$  
is definable in $\N_{normal}$, while it is not clear
 for the author whether unpairing functions for
$\langle \cdot , \cdot \rangle$ can be defined in $\N_{normal}$.
Therefore, we instead employ the following paring function.

For the function $\times$ in Example \ref{exam_snrn}.\ref{times},
$\times (x,y ; a) = y \cdot 2^{|x|} + a =
 y \underbrace{0 \dots 0}_{|x| - |a|} a
$
if $|a| \leq |x|$.
Hence we define a pairing function $\pi \in \N^{3,0}$ by
\[
 \pi (x, y_0, y_1;) = \times (x, y_0; y_1),
\]
which works as a pairing function if $|y_1| \leq |x|$.
For the function $\text{\d{--}}$ in Example \ref{exam_snrn}.\ref{-}, 
the corresponding unpairing function $\pi_0$ is defined by
\[
 \pi_0 (x, y;) = \text{\d{--}} (x; y),
\]
which works as 
$\pi_0 (x, \pi (x, y_0, y_1;); ) = y_0$.
We mention that the deletion function $D$ is needed only to define
$\pi_0$.
And another one $\pi_1$ is defined by
\[
 \pi_1 (x, y;) = y \minus \pi_0 (x, y; ) \cdot 2^{|x|}.
\]
which works as 
$\pi_1 (x, \pi (x, y_0, y_1;); ) = 
 \pi (x, y_0, y_1;) \minus y_0 \cdot  2^{|x|} = y_1
$.

Let $p ( \vec x, \vec y)$ be a polynomial such that
\[
 \max \{ |f_i ( \vec y, \vec x; \vec g ( \vec x;))|:
         i=1, \dots, l 
      \}
 \leq p(| \vec x|, | \vec y|).
\]
The existence of the polynomial $p ( \vec x, \vec y)$ is guaranteed by
 Lemma \ref{lem_lh}.
Then, $\pi ( 2^{p (| \vec x|, | \vec y|)}, z_0, z_1; )$ works as a
 pairing function for 
$f_i ( \vec y, \vec x; \vec g ( \vec x;)) \ (i=1, \dots, l)$,
and $\pi_i ( 2^{p (| \vec x|, | \vec y|)}, z;) $ 
$(i=0, 1)$ as the corresponding unpairing functions.
Recall Example \ref{exam_snrn}.\ref{exp3} here.
Since $2^{p (|\vec x|, | \vec y|)}$ is defined in $\N_{normal}$,
we thus conclude
$ f_i ( \vec y, \vec x; \vec g ( \vec x; )
     ) \in \N^{k+k', 0}
$
for each $i=1, \dots, l$.
\end{proof}

\begin{thm} 
{\normalfont (Cf. \cite{arai_eguchi} Theorem 4.1)}
 $ \fps \subseteq \N_{normal}$.
\end{thm}

\begin{proof}
Assume that $f \in \fps$ and the arity of $f$ is $k$.
Let $\vec x = x_1, \dots, x_k$ be inputs.
Then the number of possible configurations is bounded by
$2^{p (|\vec x|)}$ for some polynomial $p ( \vec x)$.
Hence $f ( \vec x)$ is computed within a number of steps bounded by
$2^{p (|\vec x|)}$, since the same configuration does not repeat in any
 terminating computations.
As in the proof of Theorem 3.9 in \cite{handley},
actions of the Turing machine can be simulated by a register machine 
working over the unary representation as follows.
\begin{itemize}
\def\labelitemi{--}
\item One register contains the number representing the
      tape-configuration to the left of the reading head, and
\item another register contains the number representing the
      tape-configuration to the right of the reading head.
\end{itemize}
We consider the following (unlimited) register machines.
A register machine $\max |\vec m|$ has registers $R_0, R_1, \dots $ which store natural
 numbers $r_0, r_1, \dots$. 
A register machine program is a finite list 
$\{ I_j: j=1, \dots, j_M\}$ of instructions.
\begin{itemize}
\item Each instruction has one of four basic types:
      (Zero) $Z(k)$; $r_k = 0$, (Successor) $S(k)$; $r_k = r_k +1$,
      (Predecessor) $P(k)$; $r_k = r_k \dot- 1$,
      (Transfer) $T(k, l)$; $r_k = r_l$,
      (Jump) $J(k, l, j)$; if $r_k = r_l$ then go to instruction $I_j$
      else go to next instruction.
\item When a computation starts, the inputs $x_1, \dots, x_k$ are stored
      in registers $R_1, \dots, R_k$, respectively, and $0$ in all the
      other registers.
\item When the computation halts, the output is the number in register $R_0$.
\end{itemize}
Functions like
$x \mapsto 2x$, $x \mapsto 2x +1$ and 
$x \mapsto \lfloor x/2 \rfloor$
are computed by a register machine within a number of steps bounded
 by the exponential of a polynomial in $|x|$.
Therefore, $f ( \vec x)$ is computed by a register machine $M$ within 
$2^{q (|\vec x|)}$-steps for some polynomial $q ( \vec x)$.

Now we simulate this computation in $\N$.
For a constant $l \geq 3+k$ depending on $M$, 
information on $M$ in step $|y|$ of the computation on inputs
$\vec x$ is represented by some functions
$f_j' (y, \vec x;)$
$(j=1, \dots, l)$, which are defined by
simultaneous safe recursion on notation of the form
\begin{equation*}
 \begin{cases}
 f'_j (0, \vec x;) = g_j ( \vec x;),
 & \text{} \\
 f'_j (C_i (y;), \vec x;) = h_j (; f_1 (y, \vec x;), \dots, f_l (y, \vec x;))
 & \text{} (i=0, 1)
 \end{cases}
\end{equation*}
for each $j=1, \dots, l$.
For the sake of completeness, we give an outline of this arithmetization.
Let us assume a (finite) encoding $\lceil I \rceil$ for an instruction $I$.
We can define the following functions in $\N_{normal}$:
\begin{eqnarray*}
index(y, \vec x;) &=& \mbox{the index of the instruction performed next}, \\
inst(y, \vec x;) &=& \mbox{a code } \lceil I \rceil 
                    \mbox{ of the instruction }
                    I \mbox{ performed next}, \\
r_j (y, \vec x;) &=& \mbox{the number } r_j \mbox{ stored in the register }
 R_j \mbox{ in step } |y|. \\ 
&& (j=0, \dots, l-3)
\end{eqnarray*}
Namely, 
$f'_1 = index$, $f'_2 = inst$, and $f'_j = r_{j-3}$ for each 
$j=3, \dots, l$.
For the base case of recursion,
$index(0, \vec x;) = 1$, 
$inst(0, \vec x;) = \lceil I_1 \rceil$, and
\begin{equation*}
r_j (0, \vec x;) =
  \begin{cases}
  0 & \text{if $j=0$,} \\
  x_j & \text{if $1 \leq j \leq k$,} \\
  0 & \text{otherwise.}
  \end{cases}
\end{equation*}
The recursion step is of the form
\begin{eqnarray*}
index(C_i (y;), \vec x;) &=& \delta_1 (; index(y, \vec x;), inst(y, \vec x;)),
 \\
inst(C_i (y;), \vec x;) &=& \delta_2 (; index(y, \vec x;), inst(y, \vec x;)),
 \\
r_j (C_i (y;), \vec x;) &=& \delta_{3+j} (; inst(y, \vec x;), 
                                      r_0 (y, \vec x;), 
                                      r_1 (y, \vec x;), \dots, 
                                      r_l (y, \vec x;)),
\end{eqnarray*}
where $\delta_1, \dots, \delta_l$ are determined according to the
 program, which depend also on the encoding
$\lceil \cdot \rceil$ for the instructions.
Obviously, each $\delta_j$ is defined only on safe arguments from $O$, $S$, $P$, $I$, $C$ by a
 suitable number of application of safe composition.
For the same reason, the above recursion step does not depend on 
$i =0, 1$ of $C_i (y;)$.

By the output convention, 
$f_3 (t, \vec x;) = r_0 (t, \vec x;) = f( \vec x)$ if
$|t| \geq 2^{q (| \vec x|)}$.
Next, via the functions $h_1, \dots, h_l$, 
we define functions $f_j (y, \vec x; \vec a) \ (j=1, \dots, l)$
by the following equations of simultaneous SNRN:
\begin{equation*}
 \begin{cases}
 f_j (0 ; \vec a) = h_j (; \vec a),
 & \text{} \\
 f_j (C_i (y;); \vec a) =
 f_j (y; f_1 (y; \vec a), \dots, f_l (y; \vec a))
 & \text{$(i= 0, 1)$} \\
 \end{cases}
\end{equation*}
Lemma \ref{lem_simul} ensures that
$F_j (y, \vec x; ) := f_j (y; g_1 ( \vec x; ), \dots, g_l ( \vec x; ))
 \in \N_{normal}
$
for each $j = 1, \dots, l$.
By the definition, the functions
$F_j (y, \vec x;) 
$
$(j=1, \dots, l)$ represent information on $M$ in step $2^{|y|}$.
Therefore,
$F_3 ( t, \vec x;) = r_0 ( 2^t, \vec x;) = f( \vec x)$
whenever $t \geq 2^{q (| \vec x|)}$.
Thus an application of safe composition yields that
$f( \vec x) = 
 F_3 ( 2^{q (| \vec x|)}, \vec x;) \in \N_{normal}
$.
\end{proof}

We mention that the formation of $F_3 ( 2^{q (| \vec x|)}, \vec x;)$ is
 now easy
 contrary to the construction in the proof of Theorem 4.1 in
 \cite{arai_eguchi}. 
This is due to the formulation of safe composition.
A more restrictive composition of the form
$f( \vec x; \vec a) = h ( \vec I^{k,0}_j ( \vec x;); 
                          \vec \varphi ( \vec x; \vec a)
  )
$ 
for projection functions $I^{k,0}_{j_1}, \dots, I^{k,0}_{j_k}$
is necessary in \cite{arai_eguchi}.

\section*{Conclusions and final comments}
In this paper we give a term-rewriting characterization of the polyspace
functions.
As a consequence, a recursion-theoretic characterization of the
polyspace functions is obtained.
This suggests a uniform approach to the space complexity in the sense
that
adding specific safe recursion schemes to a suitable base class capture
various space complexity classes (cf. Remark \ref{rem_uniform}):
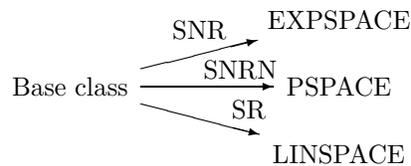
\[
 \dgARROWLENGTH=5em
 \begin{diagram}
\node[3]{\mathrm{EXPSPACE}}\\
\node{\mathrm{Base \ class}}\arrow[1]{ene,t,3}{\mbox{SNR}}
\arrow[2]{e,t,3}{\mbox{SNRN}}
\arrow[1]{ese,t,3}{\mbox{SR}}
\node[2]{\mathrm{PSPACE}}\\
\node[3]{\mathrm{LINSPACE}}
\end{diagram}
\]

We introduce a two-sorted function class $\N^{k, l}$ such that
$\N_{normal} = \fps$.
The present work is inspired by Oitavem's works, which also give a
term-rewriting and recursion-theoretic characterization of $\fps$.
However, it is not clear for the author how to prove the equality of $\N$
and the Oitavem class.
Nor proper inclusion relation between $\N$ and $\mathcal N$ in
\cite{arai_eguchi} which captures the exptime functions.
For example, it seems quite difficult to simulate safe nested recursion
within Oitavem's formulation as well as the converse is.
These are future works.

The research in such a framework has been initiated by Cichon and
Weiermann \cite{cich_weier}, and followed by Beckmann and Weiermann \cite{bec_weier}.
The underlying idea in \cite{cich_weier} and \cite{bec_weier} is that
term rewriting can be used to prove inclusion relations between complexity classes.
In particular, non-trivial closure properties are obtained.
An application in this direction will be to prove Savitch's Theorem in the
present context (as suggested by the referee).

Clearly, the present result is related to a work \cite{marion01} by Bonfante,
Marion and Moyen.
In \cite{marion01} it is shown that if a program is polynomially
quasi-interpretable and its termination is proved via the lexicographic
path order LPO, then the program is polyspace-computable.
This together with the proof of Lemma \ref{lem_termination} suggests that the TRS
$R_{\N}$ can be polynomially quasi-interpreted
(assuming a suitable rewriting strategy).
  
\section*{Acknowledgments}
First of all, I would like to acknowledge Georg Moser.
He called my attention to Oitavem's works \cite{oitavem97, oitavem02}
and a related work \cite{avanzini} by M. Avanzini.
He also pointed out some errors in an earlier draft.
I would like to thank the referee for careful reading and invaluable
comments.
In particular, the introduction of the non-confluent TRS $R_\N$ and the use
of the innermost strategy are due to him or her.
I would also like to thank my thesis advisor Prof. Toshiyasu Arai for
his comments on this work.



\end{document}